\documentclass[12pt]{amsart}
\usepackage{color}
\usepackage{amsxtra}
\usepackage{mathtools}
\usepackage{enumerate}
\usepackage{blindtext}
\usepackage[inline]{enumitem}
\usepackage{nicefrac}
\mathtoolsset{showonlyrefs} %	enumera solo las ecuaciones que se citan
\usepackage{ mathrsfs }
\usepackage{tikz}
\tikzset{
  treenode/.style = {align=center, inner sep=0pt, text centered,
    font=\sffamily},
  arn_r/.style = {treenode, circle, blue, draw=blue,fill=yellow, 
    text width=1em, very thick},% arbre rouge noir, noeud rouge
    dot/.style={circle,draw,inner sep=1.2,fill=black},
}
\usepackage{forest}

\usetikzlibrary{shadings}
\usepackage{color}
\usepackage{hyperref}
\hypersetup{
  colorlinks   = true, %Colours links instead of ugly boxes
  urlcolor     = blue, %Colour for external hyperlinks
  linkcolor    = blue, %Colour of internal links
  citecolor   = red %Colour of citations
}
\xdefinecolor{darkgreen}{RGB}{175, 193, 36}

\usepackage{amsfonts}
\usepackage{enumerate}
\usepackage{url}
\usepackage{color}      % Need the color package
\usepackage{epsfig}
\usepackage{graphicx}
\usepackage{esint}

\DeclareRobustCommand{\CChi}{{\mathpalette\irchi\relax}}
\newcommand{\irchi}[2]{\raisebox{\depth}{\large{$#1\chi$}}} % inner command, used by 
\newcommand{\T}{\mathbb{T}_m}
\renewcommand{\S}{\mathcal{S}}

\newcommand{\C}{\mathcal{C}}

\newtheorem{teo}{Theorem}[section]
\newtheorem{lem}[teo]{Lemma}
\newtheorem{co}[teo]{Corollary}

\theoremstyle{remark}
\newtheorem{re}[teo]{Remark}

\theoremstyle{definition}

\def\N{{\mathbb N}}
\def\C{{\mathbb{C}}}

\title[Convex envelopes on Trees]{Convex envelopes on Trees}
\author[L. M. Del Pezzo, N. Frevenza and J. D. Rossi]
{Leandro M. Del Pezzo, Nicolas Frevenza and Julio D. Rossi}

\address{Leandro M. Del Pezzo and Julio D. Rossi
\hfill\break\indent
CONICET and Departamento de Matem{\'a}tica, FCEyN,
Universidad de Buenos Aires,
\hfill\break\indent Pabellon I, Ciudad Universitaria (1428),
Buenos Aires, Argentina.}

\address{Nicol{\'a}s Frevenza
\hfill\break\indent
Departamento de M{\'e}todos Cuantitativos, FCEA,
Universidad de la Rep{\'u}blica,
\hfill\break\indent Gonzalo Ram{\'i}rez 1926 (11200),
Montevideo, Uruguay.}

\email{{\tt ldpezzo@dm.uba.ar, nfrevenza@dm.uba.ar, jrossi@dm.uba.ar}}

\begin{document}

\begin{abstract} 
    We introduce two notions of convexity for an infinite regular 
    tree. For these two notions we show that given a continuous
    boundary datum there exists a unique convex envelope on the 
    tree and characterize the equation that this envelope 
    satisfies. We also relate the equation with two versions 
    of the Laplacian on the tree. Moreover, for a function defined 
    on the tree, the convex envelope turns out to be
    the solution to the obstacle problem for this equation.
\end{abstract}

%\today

\maketitle

%%%%%%%%%%%%%%%%%%%%%%%%%%%%%%%%%%%%%%%%%%%%%%%%%%%%%%%%%%%%%%%%%%%%%%%%%%%%%%
\section{Introduction}
%%%%%%%%%%%%%%%%%%%%%%%%%%%%%%%%%%%%%%%%%%%%%%%%%%%%%%%%%%%%%%%%%%%%%%%%%%%%%%

In mathematics, a real-valued function defined on an interval is called convex if the line segment between any two points on the graph of the function lies above or on the graph. Equivalently, a function is convex if its epigraph (the set of points on or above the graph of the function) is a convex set. For a twice differentiable function of a single variable, if the second derivative is always greater than or equal to zero in the entire interval then the function is convex.

Convex functions play an important role in many areas of mathematics. They are especially important in the study of optimization problems where they are distinguished by a number of convenient properties. For instance, a (strictly) convex function has no more than one minimum. Even in infinite-dimensional spaces, under suitable additional hypotheses, convex functions continue to satisfy such properties and as a result, they are the most well-understood functionals in the calculus of variations. In probability theory, a convex function applied to the expected value of a random variable is always less than or equal to the expected value of the convex function of the random variable. 

On the other hand, linear and nonlinear equations (coming mainly from mean value properties) on trees are models that are close (and related to) 
to linear and nonlinear PDEs in the unit ball of $\mathbb{R}^N$, therefore, it seems natural to look for convex functions
on trees. This is our main goal in this paper.
For other analytical issues on discrete structures (including graphs such 
    as trees) we refer to \cite{ary,BBGS,DPMR1,DPMR2,DPMR3,KLW,KW,Ober,s-tree,s-tree1} and references therein.

	Let us begin by making precise the well-known notion of convexity in $\mathbb{R}^N.$
	We fix a bounded smooth domain $\Omega \subset {\mathbb{R}}^N$.
	A function $u\colon\Omega \to {\mathbb{R}}$ is said to be  a convex function 
	if for any two points $x,y \in \Omega$ such that the segment $[x,y]$ is contained in $\Omega$, 
	it holds that
	\begin{equation} \label{convexo-usual}
		u(tx+(1-t)y) \leq tu(x) + (1-t) u(y), \qquad \forall t \in (0,1).
	\end{equation}
	With this definition one can define the convex envelope of a boundary
	datum $g \colon \partial \Omega \to  {\mathbb{R}}$ as
	\begin{equation} \label{convex-envelope-usual}
		u^* (x) \coloneqq \sup \left\{u(x)\colon u \text{ is convex and verifies } u|_{\partial \Omega} \leq g \right\}.
		\end{equation}
		Here by $u|_{\partial \Omega} \leq g,$ we understand 
		\[
			\limsup_{\Omega \ni x\to z}u(z)\le g(z)
			\quad\forall z\in\partial\Omega.
		\]

	This convex envelope turns out to be the largest solution to
	\begin{equation} \label{convex-envelope-usual-eq}
		\begin{array}{ll}
			\lambda_1 (D^2 u) (x) = 0 \qquad &x\in  \Omega, 
		\end{array}
		\end{equation}
		(the equation has to be
	interpreted in viscosity sense) with 
		\[
		u(x) \leq g (x)  \qquad x\in  \partial \Omega.
		\]
	Here $\lambda_1\leq \lambda_2\leq\cdots\leq\lambda_N$ are the ordered eigenvalues of
	the Hessian matrix, $D^2u$. We refer to \cite{BlancRossi,HL1,OS,Ober33}.
	Notice that the equation 
	\[
		\lambda_1 (D^2 u) (x) = 0 
	\]
	is equivalent to 
	\[
		\min \left\{\langle D^2 u(x) v, v \rangle\colon v\in\mathbb{R}^N 
		\text{ such that } \|v\|=1\right\} =0.
	\]
	 This says that the equation that governs the convex envelope
	is just the minimum among all possible directions of 
	the second derivative of the function at 
	$x$ equal to zero. Here we notice that we have existence of a continuous 	up to the boundary 
	convex envelope for every continuous data if and only if the domain is 			strictly convex, see \cite{BlancRossi,HL1,OS}.

	In this paper, our main goal is to develop similar ideas and concepts on 	a tree. 
	When one wants to expend the notion of convexity to an ambient space 			beyond
	the Euclidean setting the two key ideas are to introduce what is a ``{}segment"{}
	in our space and once this is done, to understand what is a ``{}midpoint"{} in the segment.
	We introduce two different definitions of segments and midpoints and study the associated
	notion of convexity linked to each of them. In particular, for both definitions we are able to 
	characterize the related equation that governs the convex envelope of a given boundary
	datum.  
	
	Closely related to our results is \cite{BKND} where the authors considered convex functions
	on finite trees and showed that some relevant functions that are naturally related to spectral
	problems on the tree are convex. 
	
	Before starting with our main goal, we need to introduce our 
	ambient space.
	Given $m\in\mathbb{N}_{\ge2},$ a tree $\T$ with regular 
	$m-$branching is an infinite graph that consists of
    the empty set $\emptyset$ and all finite  sequences 
    $(a_1,a_2,\dots,a_k)$ with $k\in\N,$ 
    whose coordinates $a_i$ are chosen from $\{0,1,\dots,m-1\}.$ 
    
    \begin{center}
        \pgfkeys{/pgf/inner sep=0.19em}
        \begin{forest}
            [$\emptyset$,
                [0
                    [0
                        [0 [,edge=dotted]]
                        [1 [,edge=dotted]]
                        [2 [,edge=dotted]]
                    ]
                    [1
                        [0 [,edge=dotted]]
                        [1 [,edge=dotted]]
                        [2 [,edge=dotted]]
                    ]
                    [2
                        [0 [,edge=dotted]]
                        [1 [,edge=dotted]]
                        [2 [,edge=dotted]]
                    ]
                ]
                [1
                    [0
                        [0 [,edge=dotted]]
                        [1 [,edge=dotted]]
                        [2 [,edge=dotted]]
                    ]
                    [1
                        [0 [,edge=dotted]]
                        [1 [,edge=dotted]]
                        [2 [,edge=dotted]]
                    ]
                    [2
                        [0 [,edge=dotted]]
                        [1 [,edge=dotted]]
                        [2 [,edge=dotted]]
                    ]
                ]
                [2
                    [0
                        [0 [,edge=dotted]]
                        [1 [,edge=dotted]]
                        [2 [,edge=dotted]]
                    ]
                    [1
                        [0 [,edge=dotted]]
                        [1 [,edge=dotted]]
                        [2 [,edge=dotted]]
                    ]
                    [2
                        [0 [,edge=dotted]]
                        [1 [,edge=dotted]]
                        [2 [,edge=dotted]]
                    ]
                ]    
            ]
        \end{forest}
        
       A tree with $3-$branching.
    \end{center}
    
    The elements in $\T$ are called vertices. 
    Each vertex $x$ has $m$ successors, obtained by adding 
    another coordinate. We will denote by 
    \[
        \S(x)\coloneqq\{(x,i)\colon i\in\{0,1,\dots,m-1\}\}
    \]
    the set of successors of the vertex $x.$ 
    If $x$ is not the root then $x$ has a only an 
    immediate predecessor, which we will denote $\hat{x}.$
    The segment connecting a vertex $x$ with $\hat{x}$ is called an 
    edge and denoted by $(\hat{x},x).$
     
    A vertex $x\in\T$ has level $k\in\mathbb{N}$ if $x=(a_1,a_2,\dots,a_k)$.   
    The level of $x$ is denoted by $|x|$ and 
    the set of all $k-$level vertices is denoted by $\T\!\!\!^k.$
    We say that the edge $e=(\hat{x},x)$ has $k-$level if 
    $x\in \T\!\!\!^k.$
        
    A branch of $\T$ is an infinite sequence of vertices, where each one of them is followed 
    by one of its immediate successors.
    The collection of all branches forms the boundary of $\T$, denoted 
    by $\partial\T$.
    Observe that the mapping $\psi:\partial\T\to[0,1]$ defined as
    \[
        \psi(\pi)\coloneqq\sum_{k=1}^{+\infty} \frac{a_k}{m^{k}}
    \]
    is surjective, where $\pi=(a_1,\dots, a_k,\dots)\in\partial\T$ and
    $a_k\in\{0,1,\dots,m-1\}$ for all $k\in\mathbb{N}.$ Whenever
    $x=(a_1,\dots,a_k)$ is a vertex, we set
    \[
        \psi(x)\coloneqq\psi(a_1,\dots,a_k,0,\dots,0,\dots).
    \]

	At this point, we just have to recall that as we have mentioned,
	the definition of a convex function depends on what 
	we understand by a segment and how to define the midpoint of the segment.
	
	A path from a vertex $x$ to a vertex $y$ in $\T$ is a sequence of 
	vertices $x_0, x_1, x_2, \dots, x_k$ such that 
	$x_0= x$, $x_k=y_0$ and for any $i= 1, 2, \dots, k$ we have that
	either $\hat{x}_{i-1}=x_i$ or $x_i\in\mathcal{S}({x}_{i-1})$, that is, $x_i$ and $x_{i+1}$ are adjacent 
	(connected)
	in the graph $\T$. 
	A path is called a minimal path if and only if all the vertices
	on the path are different. Observe that for any $x,y\in\T$ there is
	a unique minimal path from $x$ to $y.$ This minimal path is  
	denoted by $[x,y].$ This is our first idea of what is a segment of $\T$. 
	 \begin{center}
        \pgfkeys{/pgf/inner sep=0.19em}
        \begin{forest}
            [$\emptyset$,circle,fill=yellow,draw=blue,
                [0,circle,fill=yellow,draw=blue[0[0][1][2]][1[0][1][2]]
                [2,circle,fill=yellow,draw=blue[0]
                [$x$,circle,fill=yellow,draw=blue][2]]]
                [1,circle,fill=yellow,draw=blue[0,circle,fill=yellow,draw=blue[$y$,circle,fill=yellow,draw=blue][1][2]][1[0][1][2]][2[0][1][2]]]
                [2[0[0][1][2]][1[0][1][2]][2[0][1][2]]]
            ]
        \end{forest}
        
      A path from a vertex $x$ to a vertex $y$.
    \end{center}

	In a slight abuse of notation, we say that an edge $e$
	belongs to a path $\gamma$ if there is a vertex $x\in\gamma$ 
	such that $e=(\hat{x},x).$ 
	We now define the length of an edge $e$ as follows:
	\[
		\text{length} (e) \coloneqq \frac{1}{m^k} 
		\qquad \text{ if } e \text{  has level } k.
	\]
	With this length we can define the total length of a path $\gamma$ as the sum 
	of the lengths of the edges involved in $\gamma$, that is,
	\[
		\text{lenght} (\gamma) \coloneqq
		\sum_{e\in\gamma} \text{length} (e).
	\]
	Now, with this notion of length of an edge and then of a path, let us 	
	consider the distance between nodes given by the length 
	of the minimal path. 
	That is, given two nodes $x$, $y$ we let
	\[
		d(x,y) \coloneqq 
		\text{lenght}([x,y]).
	\]
	Remark that $d$ is a genuine metric since $d(x,y)\ge0,$
	$d(x,y)=0$ iff $x=y$ and the triangular inequality holds (since
	the infimum of the lengths of the paths that joins $x$ with $y$ 
	is less or equal than the infimum of the length of the paths 
	that joins $x$ with $z$ plus the infimum of the length of the paths 
	that joins $z$ with $y$). 
	
	We are now ready to introduce our first notion of convex function. We say that a function $u:\T \mapsto \mathbb{R}$ is convex if
	for any $x,y,z\in \T$ with $z\in[x,y],$  it holds that
	\[
		u(z) \leq \frac{d(y,z)}{d(x,y)} u(x) + 
		\frac{d(x,z)}{d(x,y)} u(y). 
	\]
	
	Using this definition, we can look for the convex envelope 
	of a function defined on $\partial \T$. Given a function $g\colon[0,1]\to\mathbb{R},$ 
	we define 
	the convex envelope of $g$ on $\T$ as follows
	\begin{equation} \label{convex-envelope-arbol}
			u^*_g (x) \coloneqq 
			\sup \Big\{u(x) \colon u\in\mathcal{C}(g) \Big\},
	\end{equation}
	where
	\[
	    \mathcal{C}(g)\coloneqq
	    \left\{ u\colon\T\to\mathbb{R}\colon 
			u \text{ is a convex function} \text{ and } 
			\limsup_{x\to \pi\in \partial \T} u(x)\leq g (\psi(\pi)) 
		\right\}.
	\]
	The convex envelope is unique (this follows easily since the maximum 
	of two convex functions is also convex). Moreover, 
	associated to this convex envelope we have a nonlinear equation that is verified on the whole tree.
	
	\begin{teo} \label{teo.convex.arbol} 
		The convex envelope of a continuous function $g\colon[0,1]\to\mathbb{R}$ 
		is the largest solution to
		\begin{equation} \label{eq.tree.convex}
			u (x)  = \min 
			\left\{ 
					\min_{\substack{y,z\in\mathcal{S}(x)\\ y\neq z}}
					\frac{u(y) +u(z)}2  ; 
					\min_{y\in\mathcal{S}(x)}  
					\frac{ u(\hat{x}) + m \, u(y)}{m+1} 
			\right\}\quad\text{on }\T,
		\end{equation}
		that verifies 
		\begin{equation}\label{eq:bordecond}
				\limsup_{x\to \pi\in \partial \T}u(x) \leq g (\psi(\pi)).
		\end{equation}
	\end{teo}	
	
	Let us clarify that in the case $x=\emptyset$, relation 
	\eqref{eq.tree.convex} is understood as 
	\[
	    u (x)=
	    \min_{\substack{y,z\in\mathcal{S}(x)\\ y\neq z}}
					\frac{u(y) +u(z)}2,
	\]
	since $\emptyset$ does not have a predecessor because it is the 
	root of $\T.$
	
	Notice that \eqref{eq.tree.convex} is a nonlinear mean value property at the nodes of the tree.
	
	Once we have characterized the convex envelope by means of being the largest solution 
	to \eqref{eq.tree.convex} that is below the datum on $\partial \T$, we want to study the 
	associated Dirichlet problem, that is, given a datum $g$ on $\partial \T$ we want to solve the equation 
	in $\T$ and find a solution that attains continuously the datum in the sense that
	\begin{equation}\label{eq:bordecond.77}
				\lim_{x\to \pi\in \partial \T}u(x) = g (\psi(\pi)).
		\end{equation}
	For this Dirichlet problem, we can show existence and uniqueness for continuous data.
	\begin{teo} \label{teo.convex.arbol.eq} 
		Given a continuos boundary datum $g$, there is a unique
		solution to \eqref{eq.tree.convex} that verifies 
		\eqref{eq:bordecond.77}. 
	\end{teo}	
	Therefore, from Theorems \ref{teo.convex.arbol}  and \ref{teo.convex.arbol.eq}, we have that for every continuous datum on $\partial \T$ the unique 
	convex envelope attains this datum with continuity, that is, \eqref{eq:bordecond.77} holds.
	Recall that in the Euclidean case we have to ask the domain $\Omega$ to be strictly convex for
	the validity of this
	continuity up to the boundary of the convex envelope.

	Notice that the equation \eqref{eq.tree.convex} can be rewritten as
	\begin{equation} \label{eq.tree.convex.derivadassegundas}
		0  = \min 
		\left\{ 
			\min_{\substack{y,z\in\mathcal{S}(x)\\ y\neq z}}
				\frac{u(y) + u(z) -2u(x)}2 ; 
				\min_{y\in\mathcal{S}(x)}  
					\frac{u(\hat{x}) + m u(y) 
					 -(m+1) u(x)}{m+1},		  
		\right\}
	\end{equation}
	and hence we identify one possible analogous to the eigenvalues of 
	the Hessian for the Euclidean case but in the case of the tree 
	\begin{equation} \label{uu}
		\left\{ 
			\frac{u(x,i) + u(x,j) -2u(x)}2 
		\right\}_{i < j} 
		\text{ and } 
		\left\{ 
			\frac{u(\hat{x}) + m u(y) 
					 -(m+1) u(x)}{m+1}
		\right\}_{y\in\S(x)}.
	\end{equation}
	Recall that for the convex envelope
	in the Euclidean space the associated equation reads as 
	\[
		\min_{1\le j\le N} \lambda_j (D^2 u)=0,
	\] 
	and compare it with \eqref{eq.tree.convex.derivadassegundas}.
	Then, recalling that in the Euclidean setting when we add the 
	eigenvalues of the Hessian we obtain the Laplacian, we obtain the 
	following versions of a Laplacian on the tree adding 
	the expressions found in \eqref{uu},
	\begin{equation} \label{eq.Laplaciano.tree}
		u(x)= 
		\frac{2}{(m+1)^2} u(\hat{x} ) +
		\frac{ m^2 +2m -1}{(m+1)^2}  \frac1{m}\sum_{y\in\S(x)} u(y).
	\end{equation}
	Notice that this is a special case of the equations (given by mean value properties) that we previously studied 
	in \cite{DPFR-DtoN} showing existence and uniqueness for the Dirichlet problem.
	
	Finally, we also study the convex envelope of a function defined on $\T$. That is, given a bounded function
    $f:\T \mapsto \mathbb{R}$ (not necessarily convex), we look for its convex envelope that is given by
    \begin{equation} \label{convex-envelope-arbol.f}
			u^*_f (x) \coloneqq 
			\sup \big\{u(x) \colon u\in\mathcal{C}(f) \big\},
	\end{equation}
	where
	\[
	    \mathcal{C}(f)\coloneqq
	    \Big\{ u\colon\T\to\mathbb{R}\colon 
			u \text{ is a convex function and }
			 u(x)\leq f (x) \ \forall x\in \T
		\Big\}.
	\]
	The convex envelope is unique (this follows easily since the maximum 
	of two convex functions is also convex). Notice that when $f$ attains a minimum this minimum coincides
	with the minimum of the convex envelope (and it is attained at the same vertices). 
	This convex envelope turns out to be the solution to
	the obstacle problem for the equation \eqref{eq.tree.convex}. For the analogous property in the Euclidean 
	setting, we refer to \cite{Ober33}.
	Recall that for the obstacle problem one important set is the coincidence set, that is given by the
set of points where $f$ and its convex envelope $u^*_f$ coincide,
\[
CS(f) =\left\{ x\in \T \, : \, f(x) = u^*_f (x)  \right\}.
\]
	
	\begin{teo} \label{teo.convex.arbol.f} 
		The convex envelope of a function $f\colon \T \to\mathbb{R}$ 
		is the solution to the obstacle problem for the equation \eqref{eq.tree.convex}. 
		That is, $u^*_f$
		is the largest solution to
		\begin{equation} \label{eq.tree.convex.ffff}
			u (x)  \leq \min 
			\left\{ 
					\min_{\substack{y,z\in\mathcal{S}(x)\\ y\neq z}}
					\frac{u(y) +u(z)}2  ; 
					\min_{y\in\mathcal{S}(x)}  
					\frac{ u(\hat{x}) + m u(y)}{m+1} 
			\right\}\quad\text{on }\T,
		\end{equation}
		that verifies 
		\[
				u(x) \leq f(x) \qquad \forall x \in \T.
		\]
		
In the coincidence set, the function $f$ verifies an inequality
\begin{equation} \label{eq.tree.convex.ffff.latengo}
			f (x)  \leq \min 
			\left\{ 
					\min_{\substack{y,z\in\mathcal{S}(x)\\ y\neq z}}
					\frac{f(y) +f(z)}2  ; 
					\min_{y\in\mathcal{S}(x)}  
					\frac{ f(\hat{x}) + m f(y)}{m+1} 
			\right\}\quad\text{on } CS(f),
		\end{equation}
		while outside the coincidence set the convex envelope, $u^*_f$, is a solution of the equation, i.e., it holds
		\begin{equation} \label{eq.tree.convex.ffff.latengo.mmmm}
			u^*_f (x)  = \min 
			\left\{ 
					\min_{\substack{y,z\in\mathcal{S}(x)\\ y\neq z}}
					\frac{u^*_f(y) +u^*_f(z)}2  ; 
					\min_{y\in\mathcal{S}(x)}  
					\frac{ u^*_f(\hat{x}) + m u^*_f(y)}{m+1} 
			\right\}\quad\text{on }\T\setminus CS(f).
		\end{equation}
		\end{teo}

	On the other hand, in the arborescence (directed) tree, i.e., the tree defined as before 
	but adding a direction to the edges in such a way that any two edges are not directed to 
	the same vertex and the root is the unique vertex that has no edge directed into it), the 
	Laplacian is defined as the mean value of the successors minus the value at the vertex, that is, 
	\[
		 \Delta u (x) \coloneqq \frac1m \sum_{y\in\S(x)} u(y) - u(x)
		\quad\forall x\in\T,
	\]
	see \cite{KLW}.
	Now, to obtain an interpretation of this Laplacian   
	as the sum of eigenvalues of the Hessian as we did before, 
	we just have to change the notion of convexity.

Now we need to introduce extra notations.
	Given $x\in\T,$ $\mathbb{T}_{2}^{x}$ denotes
	the set of all subgraphs $\mathbb{B}$ that are formed from a finite subset 
	of the vertices of $\T$ and such that $x\in\mathbb{B},$ 
	$\S(x)\cap\mathbb{B}$ has two elements and
	for any $y\in\mathbb{B}\setminus\{x\},$ 
	\begin{itemize}
	    \item $|y|>|x|;$
	    \item either $\S(y)\cap\mathbb{B}=\emptyset$  
	        or  $\S(y)\cap\mathbb{B}$ has exactly two elements.
	\end{itemize}    
	We say that $y\in \mathbb{B}$ is an endpoint
	of $\mathbb{B}$ if $\S(y)\cap\mathbb{B}=\emptyset.$ The set of
	all endpoints of $\mathbb{B}$ is denoted $\mathcal{E}(\mathbb{B}).$
	That is, $\mathbb{B}$ is just a finite binary subtree of $\T$.
	
	 \begin{center}
        \pgfkeys{/pgf/inner sep=0.18em}
        \begin{forest}
            [$\emptyset$,
                [$x$, circle,fill=green,draw =orange 
                    [0[0][1][2]]
                    [1,  circle,fill=green,draw =orange [0][1][2]]
                    [2,  circle,fill=green,draw =orange
                        [0,  circle,fill=green,draw =orange]
                            [1,  circle,fill=green,draw  =orange][2]]]
                [1[0[0][1][2]][1[0][1][2]][2[0][1][2]]]
                [2[0[0][1][2]][1[0][1][2]][2[0][1][2]]]
            ]
        \end{forest}
       
       An  element of $\mathbb{T}_{2}^{x}.$ Root: $x$. Endpoints: $(x,1),$ $(x,2,0),$ and 
       $(x,2,1).$
    \end{center}
	
	Our second notion of convexity is the following: a function $u:\T \to \mathbb{R}$ is called binary convex if
    for any $x\in\T$
    \begin{equation}\label{convex.II}
        u(x) \leq \sum_{y\in\mathcal{E}(\mathbb{B})}\dfrac1{2^{|y|-|x|}} u(y)
        \quad \forall \mathbb{B}\in\mathbb{T}_{2}^{x}.
    \end{equation}
    In this notion of convexity, a segment is $\mathbb{B}$, a finite binary subgraph of $\T$; a midpoint
    is the root of this subgraph $\mathbb{B}$ and the convexity property just says that the value of the function
    at the midpoint is less or equal than the mean value of the values at the endpoints.
        
    Associated to this new version of convexity, we have a convex envelope 
    of a bounded boundary datum $g$ that is, defined as in \eqref{convex-envelope-arbol}, that is
    we define the  binary convex envelope of $g$ on $\T$ as follows
	\[
		\tilde{u}_g(x) \coloneqq 
			\sup \left\{u(x) \colon u\in\mathcal{B}(g) 
			 \right\}
	\]
	where
	\[
	    \mathcal{B}(g)\coloneqq \left\{
	    u\colon \T\to\mathbb{R}\colon 
	    \text{ is a binary convex function on } \T, 
		\limsup_{x\to \pi\in \partial \T} u(x)\leq g (\psi(\pi)) \right\}.
	\]
	
    For this notion of binary convex envelope, we also have an equation
(simpler than with the previous notion).

    \begin{teo} \label{teo.convex.arbol.II} 
        The binary convex envelope of a bounded boundary datum $g$ 
        is the largest solution to
        \begin{equation} \label{eq.tree.convex.II}
            u(x) = 
                \min_{\substack{y,z\in\mathcal{S}(x)\\ y\neq z}}
              \frac{u(y) + u(z)}2  \quad\text{on }\T,
        \end{equation}
        that verifies \eqref{eq:bordecond}.
    \end{teo}	

    Again, when $g$ is continuous we have a unique solution 
    to the equation that attains the boundary datum continuously.
    \begin{teo} \label{teo.convex.arbol.eq.II} 
        Given a continuous boundary datum $g$, there exists a unique
        \eqref{eq.tree.convex.II} that verifies \eqref{eq:bordecond.77}.
    \end{teo}	
    
    In this case, written \eqref{eq.tree.convex.II} as
    \[
        0 = \min_{\substack{y,z\in\mathcal{S}(x)\\ y\neq z}} 
        \left\{ \frac12 
        u(y) + \frac12 u(z) - u(x) \right\},
    \]
    we can identify the analogous to the eigenvalues of the Hessian that for 
    this case are given by,
    \begin{equation} \label{uu.II}
        \left\{ 
            \frac12 u(x,i) + \frac12 u(x,j) - u(x) 
        \right\}_{i < j} .
    \end{equation}
    Then, adding the eigenvalues in \eqref{uu.II}, we obtain
    \begin{equation} \label{eq.Laplaciano.tree.II}
        0= \frac1{m}\sum_{y\in \S(x)} u(y) -  u(x).
    \end{equation}
    Notice that this is the usual Laplacian in the arborescence tree studied in \cite{KLW}.
    
    For this notion of convexity we can also deal with the problem of the convex envelope 
    of a function $f:\T \mapsto \mathbb{R}$ defined as in \eqref{convex-envelope-arbol.f}. In this case
    we also find that this convex envelope can be characterized as the solution to the obstacle problem
    for the associated equation, \eqref{eq.tree.convex.II}, and prove a theorem analogous to
    Theorem \ref{teo.convex.arbol.f} for this case. Once we have proved Theorem \ref{teo.convex.arbol.II}, the proof of this result is similar
    to the one of Theorem \ref{teo.convex.arbol.f} and hence we leave the statement and the proof to the reader.
    
    To end this introduction, let us give a natural generalization of  
    the notion of binary convexity. Given $k\in\{2,\dots,m-2\}$ and $x\in\T,$ 
    $\mathbb{T}_{2}^{x k}$ denotes the set of all subgraphs $\mathbb{B}$ that are formed from a finite subset 
	of vertices in $\T$ and such that, $x\in\mathbb{B},$ 
	$\S(x)\cap\mathbb{B}$ has exactly $k$ elements and
	for any $y\in\mathbb{B}\setminus\{x\},$ 
	\begin{itemize}
	    \item $|y|>|x|;$
	    \item either $\S(y)\cap\mathbb{B}=\emptyset$  
	        or  $\S(y)\cap\mathbb{B}$ has exactly $k$ elements.
	\end{itemize}    
    Observe that $\mathbb{T}_{2}^{x 2}=\mathbb{T}_{2}^{x}.$
    As before we denote by $\mathcal{E}(\mathbb{B})$ (the set of endpoints) the set of points 
    $y\in \mathbb{B}$ such that $\S(y)\cap\mathbb{B}=\emptyset.$ 
    
    A function $u:\T \to \mathbb{R}$ is called $k-$convex if
    for any $x\in\T$
    \[
        u(x) \leq \sum_{y\in\mathcal{E}(\mathbb{B})}\dfrac1{k^{|y|-|x|}} u(y)
        \quad \forall \mathbb{B}\in\mathbb{T}_{2}^{x k}.
    \]
        
    As before, associated to this version of convexity, we have a convex envelope 
    of a bounded boundary datum $g$ that we will call the $k-$convex envelope of
    $g.$ Following the proof of Theorems \ref{teo.convex.arbol.II} and 
	\ref{teo.convex.arbol.eq.II}, we can show that the $k-$convex envelope 
	of $g$ is the largest solution to
        \begin{equation}\label{eq.tree.convex.III}
                u(x) = 
                \min_{\substack{x_1,\dots,x_k\in\mathcal{S}(x)\\ x_i\neq x_j}}
              \frac1k\sum_{i=1}^k u(x_i)   \quad\text{on }\T,
        \end{equation}
    that verifies \eqref{eq:bordecond}. 
    Moreover, if $g$ is a continuous function then the $k-$convex envelope 
	of $g$ is the a unique solution to \eqref{eq.tree.convex.III} that verifies 
	\eqref{eq:bordecond.77}.

    In this case, written \eqref{eq.tree.convex.III} as
    \[
        0 = \min_{\substack{x_1,\dots,x_k\in\mathcal{S}(x)\\ x_i\neq x_j}}
        \left\{ \frac1k\sum_{i=1}^k 
        u(x_i)  - u(x) \right\},
    \]
    we identify the analogous to the eigenvalues of the Hessian,
    \begin{equation} \label{uu.III}
        \left\{ 
         \frac1k\sum_{i=1}^k 
        u(x,j_i)  - u(x),
        \right\}_{j_i < j_{i+1}}.
    \end{equation}
    Adding the eigenvalues in \eqref{uu.III}, we obtain again
    \eqref{eq.Laplaciano.tree.II}, the usual Laplacian on the arborescence tree.
    
    \bigskip

    {\bf Organization of the paper.}  
        In Section \ref{sect-convex}, 
        we will prove general results for convex functions;
        In Section \ref{sect-convex-envelopes}, we prove
        our main result for the convex envelope; In Sections
        \ref{section.biconvfunction} and \ref{sect-biconvex-envelopes}, 
        we extend the results of previous sections to the notion of binary convexity.

%%%%%%%%%%%%%%%%%%%%%%%%%%%%%%%%%%%%%%%%%%%%%%%%%%%%%%%%%%%%%%%%%%%%%%%%%%%%%%
\section{Convex functions}\label{sect-convex}
%%%%%%%%%%%%%%%%%%%%%%%%%%%%%%%%%%%%%%%%%%%%%%%%%%%%%%%%%%%%%%%%%%%%%%%%%%%%%%
    We begin this section showing a different characterization of convex functions.
    
    \begin{lem}\label{lema:convexeq}
        A function $u$ on the tree is convex 
        if and only if $u$ satisfies         
        \begin{equation} \label{subsol1}
            u (x)  \le \min 
			\left\{ 
					\min_{\substack{y,z\in\mathcal{S}(x)\\ y\neq z}}
					\frac{u(y) +u(z)}2  ; 
					\min_{y\in\mathcal{S}(x)}  
					\frac{ u(\hat{x}) +  m u(y)}{m+1} 
			\right\} \qquad \forall x \in \T.
        \end{equation}
    \end{lem}
    
    \begin{proof}
        Let $u$ be a convex function. Our goal is to show that \eqref{subsol1} holds.
        Given $x$ for any $y,z\in\S(x)$ with $y\neq z$
        we have that
        \[
            u(x)\le\dfrac12 u(y)+\dfrac12 u(z)
        \]
        due to the fact that $u$ is a convex function (just take $x\in[y,z],$
        $d(y,z)=\tfrac2{m^{|x|+1}},$ and
        $d(y,x)=d(z,x)=\tfrac1{m^{|x|+1}}$ in the definition of convexity).
        Then, we get 
        \[
            u (x)  \le \min_{\substack{y,z\in\mathcal{S}(x)\\ y\neq z}}
            \frac{u(y) +u(z)}2
        \]
        for any $x\in\T.$
        
        Now, given $x\in \T$ for any $y\in \S(x)$
        \[
            u(x)\le\dfrac{m}{m+1} u(y)+\dfrac{1}{m+1} u(\hat{x})
        \]
        again due to the fact that $u$ is a convex function (in this case, take $x\in[\hat{x},y],$
        $d(\hat{x},y)=\tfrac{m+1}{m^{|x|+1}},$ 
        $d(\hat{x},x)=\tfrac{1}{m^{|x|}},$
        and
        $d(y,x)=\tfrac1{m^{|x|+1}}$). Thus
        \[
            u (x)  \le  
					\min_{y\in\mathcal{S}(x)}  
					\frac{ u(\hat{x}) + m u(y)}{m+1}
		\]
        for any $x\in\T.$
        Therefore, we have that if $u$ is a convex function then
        $u$ satisfies \eqref{subsol1}.    
        
        To see the converse, let $u$  be a function defined on the tree that verifies \eqref{subsol1} at every node and $x,y\in \T.$
        We begin by analyzing the case that $[x,y]$ is a ``straight 
        line", that is the vertices $x_0,\dots,x_N$ of $[x,y]$ 
        are such that $x_0=x,$ $x_N=y,$  $\hat{x}_i=x_{i+1}$ 
        for any $i\in\{0,\dots,N-1\}.$  More precisely, first
        we show that if $[x,y]$ is a ``straight 
        line" then
        \begin{equation}
            \label{eq:auxenvol1}
            u(x_i)\le \dfrac{d(x_i,x_0)}{d(x_N,x_0)} u(x_N) +
            \dfrac{d(x_i,x_N)}{d(x_N,x_0)} u(x_0) \quad 
            \forall i\in\{0,\dots,N\}. 
        \end{equation}
        We proceed by induction in $N$. When $N=2,$ by \eqref{subsol1}, we
        have
        \[
            u(x_1)\le\dfrac{u(x_2)+mu(x_0)}{m+1}=
            \dfrac{d(x_1,x_0)}{d(x_2,x_0)} u(x_2) +
            \dfrac{d(x_1,x_2)}{d(x_2,x_0)} u(x_0) 
        \]
        since $d(x_1,x_2)=\tfrac{1}{m^{|x_1|}},$
        $d(x_1,x_0)=\tfrac{1}{m^{|x_0|}}=\tfrac{1}{m^{|x_1|+1}}$
        and $d(x_2,x_0)=\tfrac{m+1}{m^{|x_1|+1}}.$
        
        Suppose now that \eqref{eq:auxenvol1} is true for all
        straight line that have $N-1$ vertices, where $N>2.$
        Then,
       \[
            u(x_1)\le 
                \dfrac{d(x_1,x_{N-1})}{d(x_{N-1},x_0)} u(x_0)
                +\dfrac{d(x_1,x_{0})}{d(x_{N-1},x_0)} u(x_{N-1})
        \]
                and
        \[
            u(x_{N-1})\le 
                \dfrac{d(x_1,x_{N-1})}{d(x_{N},x_1)} u(x_N)
                +\dfrac{d(x_{N},x_{N-1})}{d(x_{N},x_1)} u(x_{1}).
        \]  
        Thus,
        \begin{align*}
            u(x_1)\le \dfrac{d(x_1,x_{N-1})}{d(x_{N-1},x_0)}u(x_0)
            &+\dfrac{d(x_1,x_{0})}{d(x_{N-1},x_0)}
            \dfrac{d(x_1,x_{N-1})}{d(x_{N},x_1)} u(x_N)\\
            & \qquad + \dfrac{d(x_1,x_{0})}{d(x_{N-1},x_0)}
            \dfrac{d(x_{N},x_{N-1})}{d(x_{N},x_1)} u(x_{1}).
        \end{align*}
        Therefore,
        \begin{align*}
            &d(x_1,x_{N-1})\left[d(x_1,x_{N})u(x_0)
            +d(x_1,x_0)u(x_N)\right]\\[10pt] & \ge \left[d(x_{N-1},x_0)d(x_1,x_N)-d(x_1,x_0)d(x_{N-1},x_N)
            \right]u(x_1)\\[10pt]
            &\ge\left[
                \left\{
                    d(x_{N},x_0)-d(x_N,x_{N-1})
                \right\}
                \left\{
                    d(x_N,x_0)-d(x_1,x_0)
                \right\}
                -d(x_1,x_0)d(x_{N-1},x_N)
            \right]u(x_1)\\[10pt]
            &\ge d(x_{N},x_0)
                \left[ d(x_{N},x_0)
                    -d(x_N,x_{N-1})-d(x_1,x_0)
                \right]u(x_1)\\[10pt]
            &\ge d(x_{N},x_0)d(x_{1},x_{N-1})u(x_1).
        \end{align*}
        Then 
        \[
            u(x_1)\le \dfrac{d(x_1,x_0)}{d(x_N,x_0)} u(x_N) +
            \dfrac{d(x_1,x_N)}{d(x_N,x_0)} u(x_N).
        \]
        
        In similar manner, we get
        \begin{equation}
            \label{eq:auxenvol2}
                u(x_{N-1})\le \dfrac{d(x_{N-1},x_0)}{d(x_{N},x_0)} 
                u(x_N) +\dfrac{d(x_{N-1},x_N)}{d(x_N,x_0)} u(x_N).
        \end{equation}
        
        If $z\in [x,y]\setminus\{x_0,x_1,x_{N-1},x_N\},$ by
        the inductive hypothesis and \eqref{eq:auxenvol2}, we have
        \begin{align*}
            u(z)&\le\dfrac{d(z,x_{N-1})}{d(x_{N-1},x_0)} u(x_0)
            +\dfrac{d(z,x_{0})}{d(x_{N-1},x_0)} u(x_{N-1})\\[10pt]
            &\le \dfrac{d(z,x_0)}{d(x_N,x_0)}u(x_N)
            +\dfrac{d(z,x_{N-1})d(x_N,x_0)+d(z,x_0)d(x_{N-1},x_N)}
            {d(x_{N-1},x_0)d(x_N,x_0)}u(x_0)\\[10pt]
            &\le \dfrac{d(z,x_0)}{d(x_N,x_0)}u(x_N)\\[10pt]
            &\quad +\dfrac{\left[d(z,x_{N})-d(x_{N-1},x_N)\right]d(x_N,x_0)+
            \left[d(x_N,x_0)-d(z,x_N)\right]d(x_{N-1},x_N)}
            {d(x_{N-1},x_0)d(x_N,x_0)}u(x_0)\\[10pt]
            &\le \dfrac{d(z,x_0)}{d(x_N,x_0)}u(x_N)
            +d(z,x_N)
            \dfrac{d(x_N,x_0)-d(x_{N-1},x_N)}
            {d(x_{N-1},x_0)d(x_N,x_0)}u(x_0)\\[10pt]
            &\le 
            \dfrac{d(z,x_0)}{d(x_N,x_0)}u(x_N)
            +\dfrac{d(z,x_{N})}{d(x_N,x_0)}u(x_0),
        \end{align*}
      showing that indeed \eqref{eq:auxenvol1} holds
        when $[x,y]$ is a straight line.
        
        When $[x,y]$ is not a straight line, there is
        a $z\in[x,y]$ such that $[x,z]$ and $[z,y]$ are 
        straight lines. Observe that $[x,y]=[x,z]\cup[z,y]$ and 
        $\S(z)\cap[x,y]=\{w_1,w_2\}.$ We can assume that
        $w_1\in [x,z]$ and $w_2\in[z,y].$
        
        Thus, from \eqref{eq:auxenvol1}, we have
        \begin{align*}
            2u(z)&\le u(w_1)+ u(w_2)\\
            &\le \left[\dfrac{d(w_1,x)}{d(x,z)}
                +\dfrac{d(w_2,y)}{d(y,z)}
            \right]u(z)+\dfrac{d(w_1,z)}{d(x,z)}u(x)
            +\dfrac{d(w_2,z)}{d(y,z)} u(y).
        \end{align*}
        Then,
        \[
                \dfrac{2d(x,z)d(y,z)-d(w_1,x)d(z,y)-d(w_2,y)d(z,x)}
                {d(z,x)d(z,y)}u(z)\le
            \dfrac{d(w_1,z)}{d(x,z)}u(x)
            +\dfrac{d(w_2,z)}{d(y,z)} u(y).
        \]
        Since $d(w_1,z)=d(w_2,z),$ we get
        \begin{align*}
            d(w_1,z)&\left[d(y,z)u(x)+d(x,z) u(y)\right]\\[10pt]
            &\ge
            \left[
                2d(x,z)d(y,z)-d(w_1,x)d(z,y)-d(w_2,y)d(z,x)
            \right]u(z)\\[10pt]
            &\ge
            \left\{
                2d(x,z)d(y,z)-\left[d(x,z)-d(w_1,z)\right]d(z,y)-
                \left[d(y,z)-d(w_2,z)\right]d(z,x)
            \right\}u(z)\\[10pt]
            &\ge d(w_1,z)\left[d(z,y)+d(z,x)\right] u(z)\\
            &\ge d(w_1,z)d(x,y) u(z).
        \end{align*}
        Therefore, we obtain
        \[
            u(z)\le \dfrac{d(x,z)}{d(x,y)}u(y)+\dfrac{d(y,z)}{d(x,y)}u(x).
        \]
        
        If $w\in[x,y]\setminus\{x,z,y\}$ then $w\in[x,z]$ or 
        $w\in[z,y].$ In the case that $w\in[x,z],$ since
        $[x,z]$ is a straight line we have
        \begin{align*}
            u(w)&\le \dfrac{d(x,w)}{d(x,z)}u(z)+\dfrac{d(z,w)}{d(x,z)}u(x)\\[10pt]
            &\le \dfrac{d(x,w)}{d(x,y)}u(y)+
            \left[\dfrac{d(x,w)d(y,z)}{d(x,z)d(x,y)}+\dfrac{d(z,w)}{d(x,z)}
            \right]u(x)\\[10pt]
            &\le \dfrac{d(x,w)}{d(x,y)}u(y)+
            \dfrac{d(x,w)d(y,z)+d(z,w)d(x,y)}{d(x,z)d(x,y)}u(x)\\[10pt]
            &\le \dfrac{d(x,w)}{d(x,y)}u(y)+
            \dfrac{\left[d(x,y)-d(y,w)\right]d(y,z)+
            d(z,w)d(x,y)}{d(x,z)d(x,y)}u(x)\\[10pt]
            &\le \dfrac{d(x,w)}{d(x,y)}u(y)+
            \dfrac{\left[d(y,z)+d(z,w)\right]d(x,y)-
            d(y,w)d(y,z)}{d(x,z)d(x,y)}u(x)\\[10pt]
            &\le \dfrac{d(x,w)}{d(x,y)}u(y)+
            \dfrac{d(y,w)\left[d(x,y)-d(y,z)\right]}{d(x,z)d(x,y)}u(x)\\[10pt]
            &\le \dfrac{d(x,w)}{d(x,y)}u(y)+
            \dfrac{d(y,w)}{d(x,y)}u(x).
        \end{align*}
        In the case that $w\in [z,y]$ the proof is similar.

        Therefore, we conclude that a function $u$ that verifies  
        \eqref{subsol1} is a convex function in $\T$.
    \end{proof}
    
    Our second result show that the sum of convex function is also 
    a convex function.
    
    \begin{co}\label{co:sumconv}
        Let $u,v\colon \T\to\mathbb{R}$ be convex functions. Then 
        $u+v$ is a convex function.
    \end{co}
    \begin{proof}
        Since $u,v$ are convex function, by Lemma \ref{lema:convexeq},
        for any $x\in\T$ we have that
        \begin{align*}
            &u (x) + v(x)  \le
            \min \left\{ 
					\min_{\substack{y,z\in\mathcal{S}(x)\\ y\neq z}}
					\frac{u(y) +u(z)}2  ; 
					\min_{y\in\mathcal{S}(x)}  
					\frac{  u(\hat{x}) + m u(y)}{m+1} 
			\right\}\\
			&\hspace{5cm}+\min \left\{ 
					\min_{\substack{y,z\in\mathcal{S}(x)\\ y\neq z}}
					\frac{v(y) +v(z)}2  ; 
					\min_{y\in\mathcal{S}(x)}  
					\frac{  v(\hat{x}) + m v(y)}{m+1} 
			\right\}\\
			&\le\min \left\{ 
					\min_{\substack{y,z\in\mathcal{S}(x)\\ y\neq z}}
					\frac{u(y) +u(z)}2+\frac{v(y) +v(z)}2  ; 
					\min_{y\in\mathcal{S}(x)}  
					\frac{  u(\hat{x}) + m u(y)}{m+1}+
					\frac{  v(\hat{x}) + m v(y)}{m+1} 
			\right\}.
        \end{align*}
        Therefore,  by Lemma \ref{lema:convexeq}, $u+v$ is a convex function.
    \end{proof}
    
    It is immediate to check that the constant function $u=1$ is a convex
    function such that 
    \[
        \lim_{x\to\pi}u(x)=\CChi_{[0,1]}(\psi(\pi)) \quad
        \forall\pi\in\partial\T.
    \]
    
    We now show that for any $x_0\in\T\setminus\{\emptyset\}$ 
    there is a convex function $u$ such that 
    \[
        \limsup_{x\to\pi}u(x)\le \CChi_{I_{x_0}}
        (\psi(\pi)) \quad\pi\in\partial\T.
    \]
    Here $I_{x_0}$ is the interval associated to the vertex $x_0$ of length $\tfrac{1}{m^{|x_0|}}$ 
    given by
    \[
        I_{x_0}\coloneqq\left[\psi(x_0),\psi(x_0)+\frac1{m^{|x_0|}}\right].
    \]
    Observe that for $x_0\in \T$, $I_{x_0} \cap \partial\T$ is the
    subset of $\partial\T$ consisting of all branches that pass through $x_0$.

    To find such a convex function we introduce the following set: given $x_0\in\T$, let us consider 
    \[
        \T^{x_{0}}\coloneqq
        \{x\in\T \colon |x|\ge |x_0|,I_x\subset I_{x_0}\}.
    \]
    
    \begin{lem}\label{lema:caract}
        Let $x_0\in\T\setminus\{\emptyset\}.$ 
        Then the function $u_{x_0}\colon\T\to\mathbb{R}$
        \[
            u_{x_0}(x)\coloneqq \frac{m-1}m
            \begin{cases}
                0 &\text{if } x\not\in\T^{x_0},\\
                \displaystyle\sum_{i=0}^{|x|-|x_0|}\frac1{m^i}
                &\text{if } x\in\T^{x_0},
            \end{cases}
        \]
        is a convex function such that
        \[
            \limsup_{x\to\pi}u_{x_0}(x)
            \le
            \CChi_{I_{x_0}}(\psi(\pi)) \quad\forall
            \pi\in\partial\T.    
        \]
        Moreover, 
        \[
            u_{x_0}(x)=\min 
			    \left\{ 
					\min_{\substack{y,z\in\mathcal{S}(x)\\ y\neq z}}
					\frac{u_{x_0}(y) +u_{x_0}(z)}2  ; 
					\min_{y\in\mathcal{S}(x)}  
					\frac{ u_{x_0}(\hat{x}) +  m u_{x_0}(y)}{m+1} 
			   \right\} \quad\forall x\in\T ,
		\]
        and  for any $\pi\in\T$ such that $\psi(\pi)$ is not one of the two
        endpoints of $I_{x_0}$ we have
        \[
            \lim_{x\to\pi}u_{x_0}(x)
            =\CChi_{I_{x_0}}(\psi(\pi)).    
        \]
    \end{lem}

    \begin{proof}
        Let us start by showing  that the function
        $u_{x_0}$ is convex.  By Lemma \eqref{lema:convexeq},
        it is enough to show that $u_{x_0}$ satisfies \eqref{subsol1}.
        If $x\in \T\setminus \T^{x_0}$ then there exist 
        $y,z\in \S(x)$ such that $y\neq z,$ $u_{x_0}(y)=u_{x_0}(z)=0$. 
        So, we have
        \begin{equation}
            \label{eq:caract1}
               \min 
			    \left\{ 
					\min_{\substack{y,z\in\mathcal{S}(x)\\ y\neq z}}
					\frac{u_{x_0}(y) +u_{x_0}(z)}2  ; 
					\min_{y\in\mathcal{S}(x)}  
					\frac{ u_{x_0}(\hat{x}) +  m u_{x_0}(y)}{m+1} 
			   \right\}=0=u_{x_0}(x).
        \end{equation}
        
		If $x=x_0$, then $u_{x_0}(\hat{x}_0)=0$ and 
		$u_{x_0}(y)=\tfrac{m-1}m(1+\tfrac1m)$
			for any $y\in\S(x_0).$ Therefore
		\begin{equation}
            \label{eq:caract2}
            \begin{aligned}
		       & \min 
			    \left\{ 
					\min_{\substack{y,z\in\mathcal{S}(x_0)\\ y\neq z}}
					\frac{u_{x_0}(y) +u_{x_0}(z)}2  ; 
					\min_{y\in\mathcal{S}(x_0)}  
					\frac{ u_{x_0}(\hat{x}) +  m u_{x_0}(y)}{m+1} 
			   \right\}
			   \\[10pt]
			   & \qquad =\frac{m-1}m\min\left\{
	                1+\frac1m;
	                1\right\}\\[10pt]
			        &\qquad =\frac{m-1}m\\[10pt]
			        &\qquad =u_{x_0}(x_0).
	    	\end{aligned}
		 \end{equation}
    
		Now, suppose that $x\in\T\setminus\{x_0\}$, and so,  
		\[
		    u_{x_0}(\hat{x})=\frac{m-1}m \sum_{i=0}^{|x|-1-|x_0|}\dfrac1{m^i}
		\]
		and
		\[
		    u_{x_0}(y)=\frac{m-1}m \sum_{i=0}^{|x|+1-|x_0|}\dfrac1{m^i}.
		\]
		Hence, we obtain
		\begin{equation}
            \label{eq:caract3}
            \begin{aligned}
		       \min 
			    \Bigg\{ 
					\min_{\substack{y,z\in\mathcal{S}(x_0)\\ y\neq z}}
					\frac{u_{x_0}(y) +u_{x_0}(z)}2  &; 
					\min_{y\in\mathcal{S}(x_0)}  
					\frac{ u_{x_0}(\hat{x}) +  m u_{x_0}(y)}{m+1} 
			   \Bigg\}\\[10pt]
			   &=\frac{m-1}m\min\left\{
	                \sum_{i=0}^{|x|+1-|x_0|}\dfrac1{m^i};
	                \sum_{i=0}^{|x|-|x_0|}\dfrac1{m^i}\right\}\\[10pt]
			        &=\frac{m-1}m \sum_{i=0}^{|x|-|x_0|}\dfrac1{m^i}\\[10pt]
			        &=u_{x_0}(x).
			\end{aligned}
		 \end{equation}
        
        Therefore, by \eqref{eq:caract1}, \eqref{eq:caract2} and \eqref{eq:caract3} we get that
        \[
            u_{x_0}(x)=\min 
			    \left\{ 
					\min_{\substack{y,z\in\mathcal{S}(x)\\ y\neq z}}
					\frac{u_{x_0}(y) +u_{x_0}(z)}2  ; 
					\min_{y\in\mathcal{S}(x)}  
					\frac{ u_{x_0}(\hat{x}) +  m u_{x_0}(y)}{m+1} 
			   \right\} \quad\forall x\in\T.
		\]
		Thus, $u_{x_0}$ is a convex function.
		
		Finally, we have to show that
		\[
            \limsup_{x\to\pi}u_{x_0}(x)
            \le
            \CChi_{I_{x_0}}(\psi(\pi)) \quad\forall
            \pi\in\partial\T.    
        \]
        
        {\it Case 1}. If $\pi\in\partial\T,$ $\psi(\pi)\in I_{x_0}$
        and $\psi(\pi)$ is not an endpoint of $I_{x_0}$ then
        for any sequence $\{x_k\}_{k\in\mathbb{N}}$ in $\T$ such that
        $\pi=(x_1,\dots,x_k,\dots)$, there is $k_0\in\mathbb{N}$ such that
        $x_k\in \T^{x_0}$ for all $k\ge k_0.$ 
        Then 
        \[
                u_{x_0}(x_k)=\frac{m-1}m
                \sum_{i=0}^{|x_k|-|x_0|}\frac1{m^i}\quad \forall k\le k_0.
        \]
        Thus, as $k\to\infty$ we have
        \[
            u_{x_0}(x_k)\to 1=\CChi_{I_{x_0}}(\psi(\pi)).
        \]

        {\it Case 2}. 
        Similarly, if $\pi\in\partial\T,$ $\psi(\pi)\not\in I_{x_0}$
        then for any sequence $\{x_k\}_{k\in\mathbb{N}}$ on $\T$ such that
        $\pi=(x_1,\dots,x_k,\dots)$, we get
        \[
            u_{x_0}(x_k)\to 0=\CChi_{I_{x_0}}(\psi(\pi))
        \]
        as $k\to\infty.$
        
        {\it Case 3}.
        Finally suppose that  $\pi\in\partial\T,$ $\psi(\pi)\in I_{x_0}$
        and $\psi(\pi)$ is an endpoint of $I_{x_0}.$ 
        
        In the case that $\psi(\pi)=0$ or $\psi(\pi)=1$ 
        for any sequence $\{x_k\}_{k\in\mathbb{N}}$ on $\T$ such that
        $\pi=(x_1,\dots,x_k,\dots)$, there is $k_0\in\mathbb{N}$ such that
        $x_k\in \T^{x_0}$ for all $k\ge k_0.$ Therefore
        \[
            u_{x_0}(x_k)\to 1=\CChi_{I_{x_0}}(\psi(\pi))
        \]
        as $k\to\infty.$
        
        In the case that $\psi(\pi)\not\in\{0,1\}$ there exists sequences 
        $\{x_k\}_{k\in\mathbb{N}},$ $\{y_k\}_{k\in\mathbb{N}}$ on $\T$ 
        and $k_0\in\mathbb{N}$ such that $\pi=(x_1,\dots,x_k,\dots),$
        $\pi=(y_1,\dots,y_k,\dots),$ for all $k\ge k_0$ we get
        $x_k\in \T^{x_0}$ and $y_k\in \T^{x_0}.$   Therefore,
        \begin{align*}
            &u_{x_0}(x_k)\to 1=\CChi_{I_{x_0}}(\psi(\pi)),\\
            &u_{x_0}(y_k)\to 0\le\CChi_{I_{x_0}}(\psi(\pi)).
        \end{align*}
        This fact, together with the previous cases 1 and 2, completes the proof. 
    \end{proof}

%%%%%%%%%%%%%%%%%%%%%%%%%%%%%%%%%%%%%%%%%%%%%%%%%%%%%%%%%%%%%%%%%%%%%%%%%%%%%%
\section{Convex envelopes}\label{sect-convex-envelopes}
%%%%%%%%%%%%%%%%%%%%%%%%%%%%%%%%%%%%%%%%%%%%%%%%%%%%%%%%%%%%%%%%%%%%%%%%%%%%%%
\setcounter{equation}{0}

	In this section we deal with convex functions on the tree.
	Let us start by showing that the convex envelop $u^*_g$
	of function $g\colon[0,1]\to\mathbb{R}$, defined in \eqref{convex-envelope-arbol}, is a convex function.
	
	\begin{lem}
	    \label{lema:convenv1}
	    For any function $g\colon[0,1]\to\mathbb{R},$
	    the convex envelop $u_g^*$ is a convex function.
	\end{lem}
    \begin{proof} This follows easily from the fact that the supremum of convex functions is also convex.
        Given $g\colon[0,1]\to\mathbb{R},$ for every function $u\in\mathcal{C}(g)$ 
    	it holds that
	    \[
		    u(z) \leq \frac{d(y,z)}{d(x,y)} u(x) + 
		    \frac{d(x,z)}{d(x,y)} u(y)\le
		    \frac{d(y,z)}{d(x,y)} u^*_g(x) + 
		    \frac{d(x,z)}{d(x,y)} u^*_g(y) 
	    \]
	    for any $x,y,z\in \T$ with $z\in[x,y].$ Hence we get
        \[
            u^*_g(z) \le \frac{d(y,z)}{d(x,y)} u^*_g(x) + 
		    \frac{d(x,z)}{d(x,y)} u^*_g(y) 
	    \]
	    for any $x,y,z\in \T$ with $z\in[x,y].$ 
	    Thus $u^*_g$ is a convex function. 
    \end{proof}
    
    Our second aim is to  show that if $g$ is a continuous
    function then 
    \begin{equation}
        \label{eq:limb}
        \lim_{x\to \pi\in \partial \T} u_g^* (x) = 
            g(\psi(\pi))\quad\forall\pi\in\partial\T.
    \end{equation}
    To prove this property, we need to show a comparison principle.

    \begin{lem} \label{lema.compar.convexas}
        Let $u$ and $v$ satisfy
        \begin{align}
            \label{eq.envolvente.44.compar.u}
            u(x) &\ge 
                \min 
			    \left\{ 
					\min_{\substack{y,z\in\mathcal{S}(x_0)\\ y\neq z}}
					\frac{u(y) +u(z)}2  ; 
					\min_{y\in\mathcal{S}(x_0)}  
					\frac{  u(\hat{x}_0) + m u(y)}{m+1}
			    \right\} \quad\forall x\in\T,\\
			\label{eq.envolvente.44.compar.v}
			v(x) &\le 
			    \min 
			    \left\{ 
					\min_{\substack{y,z\in\mathcal{S}(x_0)\\ y\neq z}}
					\frac{v(y) +v(z)}2  ; 
					\min_{y\in\mathcal{S}(x_0)}  
					\frac{ v(\hat{x}_0) + m v(y)}{m+1}
			\right\}\quad\forall x\in\T,
        \end{align}
        with 
        \[
            \lim_{x\to \pi\in \partial \T} 
            u (x) \geq  \lim_{x\to \pi\in \partial \T} v (x) ,
        \]
        for every $\pi\in \partial \T$. Then,
        \[
            u(x) \geq v(x) \qquad \forall x\in \T.
        \]
    \end{lem}

    \begin{proof}
        Adding a positive constant $c$ to $u,$ 
        we may assume that 
        \begin{equation} \label{estric}
                \lim_{x\to \pi\in \partial \T} u (x)  >  
                \lim_{x\to z\in \partial \T} v (x) .
        \end{equation}
        Our goal is to show that in this case  we have
        \[
            u(x) \geq v(x)  \quad \forall x\in\T
        \]
        (and then we obtain the result 
        just by letting $c\to 0$).

        We argue by contradiction and assume that
        \[
            M = \max_{x\in \T} ( v(x)-u(x) ) >0.
        \]
        Notice that the maximum is attained thanks to 
        \eqref{estric}. Also thanks to \eqref{estric},
        we have that $M$ is attained only in a finite set of nodes. Let ${x}$ be one of such nodes. 
        From \eqref{eq.envolvente.44.compar.u} and \eqref{eq.envolvente.44.compar.v} we obtain
        \begin{align*}
            M= v(x) - u(x) \le&
                \min 
			    \left\{ 
					\min_{\substack{y,z\in\mathcal{S}(x_0)\\ y\neq z}}
					\frac{v(y) +v(z)}2  ; 
					\min_{y\in\mathcal{S}(x_0)}  
					\frac{ v(\hat{x}_0) + m v(y)}{m+1}
			    \right\}\\
			    &-\min 
			    \left\{ 
					\min_{\substack{y,z\in\mathcal{S}(x_0)\\ y\neq z}}
					\frac{u(y) +u(z)}2  ; 
					\min_{y\in\mathcal{S}(x_0)}  
					\frac{  u(\hat{x}_0) + m u(y)}{m+1}
			    \right\}.\\
        \end{align*}
        From this inequality, using that
        $$
        M\geq \left(\frac{v(y) +v(z)}2 \right) - \left( \frac{u(y) +u(z)}2 \right), \qquad \forall y,z\in\mathcal{S}(x_0)\, y\neq z,
        $$
        and
        $$
        M\geq \left(\frac{ v(\hat{x}_0) + m v(y)}{m+1} \right) - \left( \frac{  u(\hat{x}_0) + m u(y)}{m+1} \right), \qquad \forall 
        y\in\mathcal{S}(x_0),
        $$
        we get
        \begin{align*}
            M \le&
                \min 
			    \left\{ 
					\min_{\substack{y,z\in\mathcal{S}(x_0)\\ y\neq z}}
					\frac{v(y) +v(z)}2  ; 
					\min_{y\in\mathcal{S}(x_0)}  
					\frac{ v(\hat{x}_0) + m v(y)}{m+1}
			    \right\}\\
			    &-\min 
			    \left\{ 
					\min_{\substack{y,z\in\mathcal{S}(x_0)\\ y\neq z}}
					\frac{u(y) +u(z)}2  ; 
					\min_{y\in\mathcal{S}(x_0)}  
					\frac{  u(\hat{x}_0) + m u(y)}{m+1}
			    \right\} \leq M.\\
        \end{align*}
        Hence, we obtain that there are two nodes $x_1$ and 
        $x_2$ connected with $x$ (one of them can be the predecessor)
        such that 
        \[
            v(x_1)-u(x_1) = M,\qquad \text{and} \qquad 
            v(x_2)-u(x_2) = M.
        \]
        Since this happens for every $x$ in the set of maximums of $v-u$ 
        and this set is finite, we obtain a contradiction that shows that
        \[
            u(x) \geq v(x),
        \]
        and proves the result.
    \end{proof}
    
    Now we will prove \eqref{eq:limb}.
    
    \begin{teo}\label{teo:limb}
        Let $g\colon[0,1]\to\mathbb{R}$ be a continuous function. 
        Then
        \[
            \lim_{x\to \pi\in \partial \T} u_g^* (x) = 
            g(\psi(\pi))
        \]
        for any $\pi\in\partial\T.$
    \end{teo}
    
    \begin{proof}
        Let us start by observing that, for any constant $c$, 
        $u\in\mathcal{C}(g)$ if only if 
        $u+c\in\mathcal{C}(g+c).$ Therefore, 
        without loss of generality, we may 
        assume that $g$ is a nonnegative function.
        
        Let $\pi_0=(y_1,\dots,y_k,\dots)\in\partial\T.$ 
        For any $n\in\T,$ there exist $z_n\in\T^n$ and 
        $k_0$ such that $\psi(y_k)\in I_{z_n}$
        for all $k\ge k_0.$
        Now taking  $c=\min\{g(t)\colon t\in I_{z_n}\}$
        and $w_{n}=cu_{z_n}$ where $u_{z_n}$ is given by
        Lemma \ref{lema:caract}, we have that
        $w_{n}$ is a convex function such that
        \[
            \lim_{x\to\pi}w_{n}(x)\le g(\psi(\pi)),
            \qquad \forall\pi
                    \in\partial\T.
        \]
        Here, we are using that $g\ge 0$. 
        Then, $w_{n}\in\mathcal{C}(g),$ and 
        therefore $w_{n}(x)\le u^*_g(x)$ for any $x\in\T.$
        In particular, $w_{n}(y_k)\le u^*_g(y_k)$
        for any $k.$ Therefore,
        \[
           \min\{g(t)\colon t\in I_{z_n}\} =
           \lim_{k\to\infty} w_{n}(y_k)\le 
           \liminf_{k\to\infty}u_g^*(y_k).
        \]
        Taking the limit as $n\to \infty,$ 
        we have
        \[
           g(\psi(\pi_0))\le \liminf_{k\to\infty}u^*(y_k)
        \]
        since $g$ is a continuous function.
        
        Moreover, taking
        \[w^{n}(x)=a(1-u_{z_n})+bu_{z_n}=a+(b-a)w_n\] where 
        $a=2\max\{g(t)\colon t\in[0,1]\}$ and
        $b=\max\{g(t)\colon t\in I_{z_n}\},$ we have that
        \begin{align*}
            &w^{n}(x)=a+(b-a)w_n(x)\\
                &=a+(b-a)\min 
			    \left\{ 
					\min_{\substack{y,z\in\mathcal{S}(x_0)\\ y\neq z}}
					\frac{w_n(y) +w_n(z)}2  ; 
					\min_{y\in\mathcal{S}(x_0)}  
					\frac{  w_n(\hat{x}_0) + m w_n(y)}{m+1}
			    \right\}\\
			    &=\max 
			    \left\{ 
					\max_{\substack{y,z\in\mathcal{S}(x_0)\\ y\neq z}}
					a+\frac{(b-a)(w_n(y) +w_n(z))}2  ; 
					\max_{y\in\mathcal{S}(x_0)}  
					a+\frac{(b-a) (w_n(\hat{x}_0) + m w_n(y))}{m+1}
			    \right\}\\
			    &=\max 
			    \left\{ 
					\max_{\substack{y,z\in\mathcal{S}(x_0)\\ y\neq z}}
					\frac{w^n(y) +w^n(z)}2  ; 
					\max_{y\in\mathcal{S}(x_0)}  
				    \frac{w^n(\hat{x}) + m w^n(y)}{m+1}
			    \right\}\\
			    &\ge\min 
			    \left\{ 
					\min_{\substack{y,z\in\mathcal{S}(x_0)\\ y\neq z}}
					\frac{w^n(y) +w^n(z)}2  ; 
					\min_{y\in\mathcal{S}(x_0)}  
				    \frac{w^n(\hat{x}) + m w^n(y)}{m+1}
			    \right\}
	    \end{align*}
        for any $x\in\T$ and
        \[
            g(\psi(\pi))\le  \liminf_{z\to\pi}w^{n}(x_k) ,
            \qquad \forall\pi\in\partial\T.
        \]
        Thus, by Lemma \ref{lema.compar.convexas},
        for any $u\in\mathcal{C}(g)$ we have that 
        $u(x)\le w^{n}(x)$ for any $x\in\T.$ 
        Therefore $u_g^*(x)\le w^{n}(x)$ for any $x\in\T.$
        In particular, $u_g^*(y_k)\le w^{n}(y_k) $
        for any $k.$ Then
        \[
            \limsup_{k\to\infty}u_g^*(y_k)\le
            \lim w_{n}(y_k)
            =\max\{g(t)\colon t\in I_{z_n}\}.
        \]
        Again, taking the limit as $n\to \infty,$ 
        we have
        \[
            \limsup_{k\to\infty}u_g^*(y_k)\le g(\psi(\pi_0)).
        \]
        
        Therefore, we conclude that
        \[
            \lim_{k\to\infty}u^*(y_k)= g(\psi(\pi_0)).
        \]
        As $\pi_0\in\partial\T$ was arbitrary, we conclude
        \[
            \lim_{x\to\pi_0}u^*(x)= g(\psi(\pi_0))
        \]
        for any $\pi_0\in\partial\T.$
    \end{proof}
    
    Now our next goal is to find the equation that $u^*_g$ verifies on $\T$.

    \begin{teo}\label{teo:largesol} 
        Let $g\colon[0,1]\to\mathbb{R}$ be a continuous
        functions. The convex envelope $u^*_g$ is characterized as the largest solution
        to
        \begin{equation} \label{eq.envolvente}
            u (x)  = \min 
			\left\{ 
					\min_{\substack{y,z\in\mathcal{S}(x)\\ y\neq z}}
					\frac{u(y) +u(z)}2  ; 
					\min_{y\in\mathcal{S}(x)}  
					\frac{  u(\hat{x}) + m u(y)}{m+1} 
			\right\}\quad\text{on }\T,
        \end{equation}
        that verifies 
        \[
            \limsup_{x\to \pi\in \partial \T} u (x) \leq 
            g(\psi(\pi)).
        \]
    \end{teo}

    \begin{proof} 
        Given $g\colon[0,1]\to\mathbb{R},$ by Lemmas \ref{lema:convenv1} and
        \ref{lema:convexeq} we get that $u^*_g$ verifies \eqref{subsol1}.

        Now, to see that we have an equality, 
        we argue by contradiction. Assume that at some node $x_0\in\T,$ 
        we have
        \[
            u_g^* (x_0)  < \min 
			\left\{ 
					\min_{\substack{y,z\in\mathcal{S}(x_0)\\ y\neq z}}
					\frac{u_g^*(y) +u_g^*(z)}2  ; 
					\min_{y\in\mathcal{S}(x_0)}  
					\frac{  u_g^*(\hat{x}_0) + m u_g^*(y)}{m+1}
			\right\}
		\]
		
		Taking $\delta>0$  such that
		\[
            u_g^* (x_0) + \delta  < \min 
			\left\{ 
					\min_{\substack{y,z\in\mathcal{S}(x)\\ y\neq z}}
					\frac{u_g^*(y) +u_g^*(z)}2  ; 
					\min_{y\in\mathcal{S}(x)}  
					\frac{  u_g^*(\hat{x}) + m u_g^*(y)}{m+1}
			\right\}
		\]
		and consider
        \[
            v(x) =
                \begin{cases}
                    u^*_g (x) &\text{if }x\neq x_0,\\
                    u_g^* (x_0) + \delta&\text{if }x= x_0.
                \end{cases}
        \]
        
        Observe that $v$ verifies \eqref{subsol1}. Thus, by
        Lemma \ref{lema:convexeq}, $v$ is convex. In addition, we have
        that $v\in\C(g).$ Therefore 
        \[
            v(x)\le u_g^*(x)\quad\forall x\in \T,
        \]
        leading to a contradiction. This proves that $u^*_g$ is a solution 
        to \eqref{eq.envolvente}.

        Finally, to see that $u^*_g$ is the largest solution 
        to \eqref{eq.envolvente} that verifies 
        \[
            \limsup_{x\to \pi\in \partial \T} u^*_g (x) \leq g(\psi(\pi)),
        \]
        it is enough to define
        \[
        \overline{u} (x) = \sup 
            \left\{ u(x)\colon u\text{ verifies \eqref{eq.envolvente} and }  \limsup_{x\to \pi\in \partial \T} u (x) 
            \leq g(\psi(\pi)) 
            \right\}.
        \]
        This function $\overline{u}$ trivially verifies
        \[
            \overline{u} (x) \geq u^*_g (x) \quad x \in \T,
        \]
        just notice that $u^*_g$ belongs to the set defining $\overline{u}$.
    
        On the other hand, since $\overline{u}$ is a solution to 
        \eqref{eq.envolvente}, by Lemma \ref{lema:convexeq}, we have that
        $\overline{u}$ is convex and therefore $\overline{u}\in\mathcal{C}(g).$
        Then
        \[
            \overline{u} (x) \leq u^*_g (x) \quad\forall x \in \T.
        \]
        We conclude that
        \[
            u^*_g(x) = \overline{u} (x) = 
            \sup \left\{ v(x) \colon  
            u \text{ verifies \eqref{eq.envolvente} and }  
            \limsup_{x\to \pi\in \partial \T} u (x) \leq g(\psi(\pi)) 
            \right\}.
        \]
    \end{proof}
    
    Observe that by Lemma \ref{lema.compar.convexas} and 
    Theorem \ref{teo:largesol}, 
    for any continuous function $g\colon [0,1] \mapsto \mathbb{R}$, the equation defining the convex envelope 
    has a unique solution that attains the datum $g$ continuously. 

    \begin{teo} 
        Let $g\colon[0,1] \mapsto \mathbb{R}$ be a continuous function. 
        There exists a unique solution to \eqref{eq.envolvente} such that 
        \[ 
            \lim_{x\to \pi \in \partial \T} u (x) = g(\psi(\pi)).
        \]
        for any $\pi\in\T.$
    \end{teo}
    
    To end this section we prove Theorem \ref{teo.convex.arbol.f} 
 that deals with the convex envelope of a function $f:\T \mapsto \mathbb{R}$ given by \eqref{convex-envelope-arbol.f}.
 
 \begin{teo} \label{teo.convex.arbol.f.sec} 
		The convex envelope of a function $f\colon \T \to\mathbb{R}$ 
		is the solution to the obstacle problem for the equation \eqref{eq.tree.convex}. 
		\end{teo}
		
		\begin{proof}
		Let us call $v^*$ to the largest solution to
		\begin{equation} \label{eq.tree.convex.ffffhhh}
			u (x)  \leq \min 
			\left\{ 
					\min_{\substack{y,z\in\mathcal{S}(x)\\ y\neq z}}
					\frac{u(y) +u(z)}2  ; 
					\min_{y\in\mathcal{S}(x)}  
					\frac{ u(\hat{x}) + m u(y)}{m+1} 
			\right\}\quad\text{on }\T,
		\end{equation}
		that verifies 
		\[
				u(x) \leq f(x) \qquad \forall x \in \T.
		\]
		
		We have to prove that the convex enevolpe of $f$, $u^*_f$, verifies
		\[
		u^*_f (x) = v^* (x), \qquad \forall x \in \T.
		\]
		
		Since $u^*_f$ is convex from Lemma \ref{lema:convexeq} we obtain that it is a solution to 
		\eqref{eq.tree.convex.ffffhhh} that verifies $u^*_f \leq f$ on $\T$ and then we obtain
		\[
		u^*_f (x) \leq v^* (x), \qquad \forall x \in \T.
		\]
		
		We have also from Lemma \ref{lema:convexeq} that $v^*$ being a solution to
		 \eqref{eq.tree.convex.ffffhhh} is a convex function and it verifies $v^* \leq f$ on $\T$. Hence,
		\[
		v^* (x) \leq u^*_f (x), \qquad \forall x \in \T.
		\]
We conclude that 
		\[
		u^*_f (x) = v^* (x), \qquad \forall x \in \T.
		\]
				
In the coincidence set, the function $f$ verifies an inequality. From the fact that $u^*_f$ is convex and smaller 
than $f$ we obtain for $x \in CS(f)$,
\begin{align} \label{eq.tree.convex.ffff.latengo.kkk}
		f (x) & = u^*_f(x) \\
		& \leq \displaystyle \min 
			\left\{ 
					\min_{\substack{y,z\in\mathcal{S}(x)\\ y\neq z}}
					\frac{u^*_f(y) + u^*_f(z)}2  ; 
					\min_{y\in\mathcal{S}(x)}  
					\frac{ u^*_f(\hat{x}) + m \,  u^*_f (y)}{m+1} 
			\right\} \\[10pt]
			& \displaystyle  \leq \min 
			\left\{ 
					\min_{\substack{y,z\in\mathcal{S}(x)\\ y\neq z}}
					\frac{f(y) + f(z)}2  ; 
					\min_{y\in\mathcal{S}(x)}  
					\frac{ f(\hat{x}) + m f (y)}{m+1} 
			\right\}.
\end{align}

		Finally, outside of the coincidence set the convex envelope, $u^*_f$, is a solution to the equation.
		In fact, arguing by contradiction, assume that for some $x_0 \not\in CS(f)$ it holds
		\begin{equation} \label{eq.tree.convex.ffff.latengo.mmmm.llll}
			u^*_f (x_0)  < \min 
			\left\{ 
					\min_{\substack{y,z\in\mathcal{S}(x_0)\\ y\neq z}}
					\frac{u^*_f(y) +u^*_f(z)}2  ; 
					\min_{y\in\mathcal{S}(x_0)}  
					\frac{ u^*_f(\hat{x_0}) + m u^*_f(y)}{m+1} 
			\right\}.
		\end{equation}
		Then, since $x_0 \not\in CS(f)$ and we have a strict inequality in \eqref{eq.tree.convex.ffff.latengo.mmmm.llll}
		there exists $\delta>0$ such that the function
		\[
            v(x) =
                \begin{cases}
                    u^*_f (x) &\text{if }x\neq x_0,\\
                    u_f^* (x_0) + \delta&\text{if }x= x_0
                \end{cases}
        \]
        is convex and still verifies $v\leq f$ on $\T$ contradicting the maximality of the convex envelope $u^*_f$.
		This contradiction shows that we have an equality in \eqref{eq.tree.convex.ffff.latengo.mmmm.llll}.
		\end{proof}

%%%%%%%%%%%%%%%%%%%%%%%%%%%%%%%%%%%%%%%%%%%%%%%%%%%%%%%%%%%%%%
\section{Binary convex functions}
\label{section.biconvfunction}
%%%%%%%%%%%%%%%%%%%%%%%%%%%%%%%%%%%%%%%%%%%%%%%%%%%%%%%%%%%%%%
    As in Section \ref{sect-convex}, we begin showing a different characterization of binary convex functions.
    
    \begin{lem}\label{lema:biconvexeq}
        A function $u$ on the tree is binary convex 
        if and only if $u$ satisfies         
        \begin{equation} \label{subsol2}
            u (x)  \le 
					\min_{\substack{y,z\in\mathcal{S}(x)\\ y\neq z}}
					\frac{u(y) +u(z)}2   
					\qquad \forall x\in \T.
        \end{equation}
    \end{lem}
    \begin{proof}
        Let us start the proof observing that if $x\in \T,$
        $y,z\in \S(x)$ and $y\neq z$ then $[y,z]\in\mathbb{T}_{2}^{x}$ and $\mathcal{E}([y,z])=\{y,z\}.$
        Therefore if $u$ is a binary convex function, 
        $x\in\T$ and $y,z\in \S(x)$ are such taht $y\neq z$
        then
        \[
            u(x)\le \dfrac{u(y)+u(z)}{2}.
        \]
        Thus, $u$ satisfies \eqref{subsol2} in $\T.$
        
        Now assume that $u$ satisfies \eqref{subsol2}. Our aim is to prove that $u$ is a binary convex function,
        that is, we aim to show that
        \begin{equation}\label{eq:biconvdef}
            u(x)\le \sum_{y\in\mathcal{E}(\mathbb{B})}
            \dfrac{u(y)}{2^{|y|-|x|}}\quad \forall 
            \mathbb{B}\in\mathbb{T}_{2}^{x}.
        \end{equation}
        
        Fix $x\in \T.$  Given $\mathbb{B}\in\mathbb{T}_{2}^{x},$ we
        define 	 
        \[
            |\mathbb{B}|\coloneqq
           \max \big\{|z|-|x|\colon z\in\mathcal{E}(\mathbb{B})\big\}\in\mathbb{N}
        \]
        and
        \[
            \mathbb{T}_{2}^{x n} \coloneqq
            \big\{\mathbb{B}\in\mathbb{T}_{2}^{x}\colon |\mathbb{B}|=n\big\}\subset\mathbb{T}_{2}^{x}.
        \]
        
        The proof of \eqref{eq:biconvdef} runs by induction in
        $n.$
        Observe that in the case $|\mathbb{B}|=1$ there exist 
        $y,z\in\S(x)$ such that $\mathbb{B}=[y,z]$ and
        obviously $\mathcal{E}(\mathbb{B})=\{y,z\}.$ Then,
        since $u$ satisfies \eqref{subsol2}, we get
        \[
            u(x)\le\dfrac{u(y)+u(z)}2=
            \sum_{y\in\mathcal{E}(\mathbb{B})}
            \dfrac{u(y)}{2^{|y|-|x|}}.
        \]
        That is \eqref{eq:biconvdef} holds for any 
        $\mathbb{B}\in\mathbb{T}_{2}^{x 1}.$
        
        Now we assume that \eqref{eq:biconvdef} holds for any 
        $\mathbb{B}\in\mathbb{T}_{2}^{x n},$ and we will show that it also holds for any $\mathbb{B}\in\mathbb{T}_{2}^{x (n+1}.$
        
        If $\mathbb{B}\in\mathbb{T}_{2}^{x (n+1)}$ then
        $\mathbb{B}'=\mathbb{B}\setminus 
        \{y\in\mathcal{E}(\mathbb{B})\colon |y|-|x|=n+1\}\in\mathbb{T}_{2}^{x n}.$ Then, 
        by the inductive hypothesis, we get
        \begin{equation}\label{eq:biconaux1}
            u(x)\le\sum_{y\in\mathcal{E}(\mathbb{B}')}
            \dfrac{u(y)}{2^{|y|-|x|}}.
        \end{equation}
        
        On the other hand, for any 
        $y\in\mathcal{E}(\mathbb{B})$ we have that  $y\in\mathcal{E}(\mathbb{B}')$ or there are 
        $w\in \mathcal{E}(\mathbb{B}')$ and 
        $z\in\mathcal{E}(\mathbb{B})\setminus\{y\}$
        such that $y,z\in \S(w).$ Thus, since
        $u$ satisfies \eqref{subsol2}, from
        \eqref{eq:biconaux1}, we have that
        \[
            u(x)\le\sum_{y\in\mathcal{E}(\mathbb{B})}
            \dfrac{u(y)}{2^{|y|-|x|}}.
        \]
        
        Finally, since $x$ is arbitrary, we conclude that $u$ is a binary convex function.
    \end{proof}
    
    \begin{re}\label{re:cfisbcf}
        Now, by Lemmas \ref{lema:convexeq} and 
        \ref{lema:biconvexeq}, it is easy to check 
        that a convex function is also a binary convex 
        function.
    \end{re}
    Proceeding as in the proof of Corollary 
    \ref{co:sumconv} we can prove the following result.
    
    \begin{co}\label{co:sumbiconv}
        Let $u,v\colon \T\to\mathbb{R}$ be binary 
        convex functions. Then $u+v$ is a 
        binary convex function.
    \end{co}
    
    Now, we obtain the following result, whose proof is similar to Lemma  
    \ref{lema:caract}. 
    \begin{lem}\label{lema:bicaract}
        Let $x_0\in\T\setminus\{\emptyset\}.$ 
        Then the function $u_{x_0}\colon\T\to\mathbb{R}$ defined by
        \[
            u_{x_0}(x)\coloneqq 
            \begin{cases}
                0 &\text{if } x\not\in\T^{x_0},\\
                1&\text{if } x\in\T^{x_0},
            \end{cases}
        \]
        is a binary convex function such that
        \[
            \limsup_{x\to\pi}u_{x_0}(x)
            \le
            \CChi_{I_{x_0}}(\psi(\pi)) \quad\forall
            \pi\in\partial\T.    
        \]
        Moreover, 
        \[
            u_{x_0}(x)= 
					\min_{\substack{y,z\in\mathcal{S}(x)\\ y\neq z}}
					\frac{u_{x_0}(y) +u_{x_0}(z)}2 \quad\forall x\in\T,
		\]
        and for any $\pi\in\T$ such that $\psi(\pi)$ is not an
        endpoint of $I_{x_0}$ we have
        \[
            \lim_{x\to\pi}u_{x_0}(x)
            =\CChi_{I_{x_0}}(\psi(\pi)).    
        \]
    \end{lem}
    
%%%%%%%%%%%%%%%%%%%%%%%%%%%%%%%%%%%%%%%%%%%%%%%%%%%%%%%%%%%%%%
\section{Binary Convex envelopes}
\label{sect-biconvex-envelopes}
%%%%%%%%%%%%%%%%%%%%%%%%%%%%%%%%%%%%%%%%%%%%%%%%%%%%%%%%%%%%%%
    Proceeding as in the proof of Lemma \ref{lema:convenv1}
    we can show that the binary convex envelop is a binary convex function.
    
	\begin{lem}
	    \label{lema:biconvenv1}
	    For any function $g\colon[0,1]\to\mathbb{R},$
	    the binary convex envelop $\tilde{u}_g$ is a 
	    binary convex function.
	\end{lem}
	
    In a similar way to Section \ref{sect-convex-envelopes}, 
    we will show that if $g$ is a continuous function, then 
    \begin{equation}
        \label{eq:bilimb}
        \lim_{x\to \pi\in \partial \T} \tilde{u}_g (x) = 
            g(\psi(\pi))\quad\forall\pi\in\partial\T.
    \end{equation}
    As before, to prove this claim we need a comparison principle.

    \begin{lem} \label{lema.bicompar.convexas}
        Let $u$ and $v$ satisfy
        \begin{align}
            \label{eq.bienvolvente.44.compar.u}
            u(x) &\ge 
                	\min_{\substack{y,z\in\mathcal{S}(x_0)\\ y\neq z}}\frac{u(y) +u(z)}2 \quad\forall x\in\T,\\
			\label{eq.bienvolvente.44.compar.v}
			v(x) &\le 
					\min_{\substack{y,z\in\mathcal{S}(x_0)\\ y\neq 
					z}}
					\frac{v(y) +v(z)}2 , \quad\forall x\in\T,
        \end{align}
        with 
        \[
            \lim_{x\to \pi\in \partial \T} 
            u (x) \geq  \lim_{x\to \pi\in \partial \T} v (x) ,
        \]
        for every $\pi\in \partial \T$. Then,
        \[
            u(x) \geq v(x) \qquad \forall x\in \T.
        \]
    \end{lem}

    \begin{proof}
        Adding a positive constant $c$ to $u,$ 
        we may assume that 
        \begin{equation} \label{biestric}
                \lim_{x\to \pi\in \partial \T} u (x)  >  
                \lim_{x\to z\in \partial \T} v (x) .
        \end{equation}
        
        We argue by contradiction, so, assume that
        \[
            M = \max_{x\in \T} (v(x)-u(x)) >0.
        \]
        Notice that the maximum is attained thanks to 
        \eqref{biestric}. Also by \eqref{biestric},
        we have that $M$ is attained only in a finite set of vertices. Let ${x}$ be one of such vertices. 
        From \eqref{eq.bienvolvente.44.compar.u} and 
        \eqref{eq.bienvolvente.44.compar.v} we obtain 
                \[
            M= v(x) - u(x) \le
					\min_{\substack{y,z\in\mathcal{S}(x_0)\\ 
					y\neq z}}\frac{v(y) +v(z)}2
			    -\min_{\substack{y,z\in\mathcal{S}(x_0)\\ y\neq z}}
					\frac{u(y) +u(z)}2. 
        \]
        Now, using that
        $$
        M\geq \left(\frac{v(y) +v(z)}2 \right) - \left( \frac{u(y) +u(z)}2 \right), \qquad \forall y,z\in\mathcal{S}(x_0)\, y\neq z,
        $$
                we get
           \[
            M \le
					\min_{\substack{y,z\in\mathcal{S}(x_0)\\ 
					y\neq z}}\frac{v(y) +v(z)}2
			    -\min_{\substack{y,z\in\mathcal{S}(x_0)\\ y\neq z}}
					\frac{u(y) +u(z)}2 \leq M. 
        \]        
        Then, there exist two nodes $x_1, x_2\in\S(x)$ such that 
        \[
            v(x_1)-u(x_1) = M,\qquad \text{and} 
            \qquad v(x_2)-u(x_2) = M.
        \]
        Since this happens for every $x$ in the set of maximums of 
        $v-u$ and this set is finite, we obtain a contradiction that 
        shows that
        \[
            u(x) \geq v(x)
        \]
        and proves the result.
    \end{proof}
    
    Now we will show \eqref{eq:bilimb}.
    
    \begin{teo}\label{teo:bilimb}
        Let $g\colon[0,1]\to\mathbb{R}$ be a continuous function. 
        Then, for any $\pi\in\partial\T$
        \[
            \lim_{x\to \pi\in \partial \T} \tilde{u}_g (x) = 
            g(\psi(\pi)).
        \]
    \end{teo}
    
    \begin{proof}
        Let us start by observing that, for any constant $c$, 
        $u\in\mathcal{B}(g)$ if only if 
        $u+c\in\mathcal{B}(g+c).$ Therefore, 
        without loss of generality, we may 
        assume that $g$ is nonnegative.
        
        Consequently, by Remark \ref{re:cfisbcf}
        and Theorem \ref{teo:limb}, we have
        \[
            \liminf_{x\to \pi\in \partial \T} \tilde{u}_g (x) \ge
            \lim_{x\to \pi\in \partial \T} u^*_g (x)=
            g(\psi(\pi))
        \]
        for any $\pi\in\partial\T.$
        
        To complete the proof, we proceed as in the end of the proof 
        of Theorem \ref{teo:limb}, using Lemmas \ref{lema:bicaract} and \ref{lema.bicompar.convexas} instead of Lemmas
        \ref{lema:caract} and \ref{lema.compar.convexas}.
    \end{proof}

    Finally, with analogous arguments of the Section \ref{sect-convex-envelopes}, we get the following results.

    \begin{teo}\label{teo:bilargesol} 
        Let $g\colon[0,1]\to\mathbb{R}$ be a continuous
        functions. The binary convex envelope $\tilde{u}_g$ 
        is characterized as the largest solution
        to
        \begin{equation} \label{eq.bienvolvente}
            u (x)  = 
					\min_{\substack{y,z\in\mathcal{S}(x)\\ y\neq z}}
					\frac{u(y) +u(z)}2 \quad\text{on }\T,
        \end{equation}
        that verifies 
        \[
            \limsup_{x\to \pi\in \partial \T} u (x) \leq 
            g(\psi(\pi)).
        \]
    \end{teo}

    \begin{teo} 
        Let $g\colon [0,1] \mapsto \mathbb{R}$ be a continuous function. Then, there exists a unique solution to \eqref{eq.bienvolvente} such that         for any $\pi\in\T$
        \[ 
            \lim_{x\to \pi \in \partial \T} u (x) = g(\psi(\pi)).
        \]
    \end{teo}

        \medskip

{\bf Acknowledgements.}  \

Supported by CONICET grant PIP GI 11220150100036CO (Argentina), by UBACyT grant 20020160100155BA (Argentina) and by MINECO MTM2015-70227-P (Spain).

%%%%%%%%%%%%%%%%%%%%%%%%%%%%%

\end{document}